\documentclass{article}
\usepackage{authblk}
\usepackage{amsmath}
\usepackage{amssymb}
\newtheorem{theorem}{Theorem}

\newenvironment{proof}{\noindent{\bf Proof:} \hspace*{1mm}}{
	\hspace*{\fill} $\Box$ \bigskip}

\newtheorem{mydef}{Definition}

\newtheorem{cor}{Corollary}

\begin{document}

\title{On a conjecture of Erd\H{o}s and Szekeres}         
\author{by Georgios Vlachos \\
		gvlachos@mit.edu}        
\date{}          
\maketitle
 

\begin{abstract}
Let f(n) denote the smallest positive integer such that every set of $f(n)$ points in general position in the Euclidean plane contains a convex n-gon. In a seminal paper published in 1935, Erd\H{o}s and Szekeres proved that f(n) exists and provided an upper bound. In 1961, they also proved a lower bound, which they conjectured is optimal. Their bounds are: $2^{n-2}+1 \leq f(n) \leq {2n - 4 \choose n-2}+1$. Since then, the upper bound has been improved by rougly a factor of 2, to $f(n) \leq {2n - 5 \choose n-2}+1$. In the current paper, we further improve the upper bound by proving that:
$$ \limsup\limits_{n\rightarrow \infty} \frac{f(n)}{{2n-5 \choose n-2}} \leq \frac{29}{32}$$ 
\end{abstract}

\section{Introduction}

In 1935, Paul Erd\H{o}s and George Szekeres published a paper titled {\em 'A combinatorial problem in geometry'} [2], in which they published classical results in Ramsey Theory. One of these results was the proof of the existence of a minimum positive integer $f(n)$ such that, for any set of $f(n)$ points in general position in the plane, there exists a subset of $n$ poitns that forms a convex $n$-gon. They presented several proofs of the existence of $f(n)$ and gave an upper bound. In [3], published in 1961, they also provided a lower bound by construction, which has been conjectured to be the exact value of $f(n)$. These bounds were:
$$2^{n-2} + 1 \leq f(n) \leq {2n -4 \choose n-2}+1$$
The upper bound was not improved for 60 years, until in 1998, Fan Chung and Ronald Graham proved that $f(n) \leq {2n-4 \choose n-2}$ [1]. Two improvements soon followed in the same year. In [4] Daniel Kleitman and Lior Pachter proved that $f(n) \leq {2n-4 \choose n-2} -2n+7$, while G\'{e}za T\'{o}th and Pavel Valtr [5] improved it to $f(n) \leq {2n-5 \choose n-2} + 2$. The best known upper bound, published in 2005 again by T\'{o}th and Valtr [6], is $f(n) \leq {2n-5 \choose n-2} +1$. \\
In the current paper, we further improve the best currently known upper bound. We describe the upper bound bound in two equivalent ways, one of which, given in the Appendix, is in the closed-form function: $$f(n) \leq {2n-5 \choose n-2} - {2n-8 \choose n-5} + {2n-10 \choose n-7}+2$$
In our main text, we describe a function that is not in closed-form, but provides a more intuitive understanding of the algebraic manipulations. Using any of these two results, we can directly prove that: $$ \limsup\limits_{n\rightarrow \infty} \frac{f(n)}{{2n-5 \choose n-2}} \leq \frac{29}{32}$$

\section{Convex Sequences}

In this section we reprove the original upper bound from [2], by slightly modifuing their argument. First, we will need some definitions, which to our knowledge were first introduced in [2]. \\

\begin{mydef} \textbf{:} \\
\textbf{n-cup:} We define a set of $n$ points in the plane to be an $n$-cup iff, when the points are listed as $(x_1,y_1),(x_2,y_2),..,(x_n,y_n)$ in increasing order of $x$ coordinates, the following holds: 
$$ \frac{y_2-y_1}{x_2-x_1} < \frac{y_3-y_2}{x_3-x_2} < .. < \frac{y_n-y_{n-1}}{x_n -x_{n-1}} $$
\textbf{n-cap:} We define a set of $n$ points in the plane to be an $n$-cup iff, when the points are listed as $(x_1,y_1),(x_2,y_2),..,(x_n,y_n)$ in increasing order of $x$ coordinates, the following holds: 
$$ \frac{y_2-y_1}{x_2-x_1} > \frac{y_3-y_2}{x_3-x_2} > .. > \frac{y_n-y_{n-1}}{x_n -x_{n-1}} $$
\end{mydef}

Clearly, $n$-cups and $n$-caps are convex $n$-gons, so the existence of an $n$-cup or an $n$-cap in a set of points implies the existence of a convex $n$-gon.  

\begin{theorem}[Erd\H{o}s and Szekeres] \textbf{:} 
Let $f(n,m)$ denote the largest positive integer, with the property that there exists a set of $f(n,m)$ points in general position in the plane that contains no n-cup and no m-cap. Then $f(n,m) = $${m+n-4}\choose{n-2}$.
\end{theorem}

\begin{proof}
When $n=3$ or $m=3$, the formula is trivially correct. If we can show that $f(n+1,m+1)=f(n+1,m)+f(n,m+1)$, then simple algebra gives the above formula. So it all boils down to proving this recurrence relation. \\
We will first prove that $f(n+1,m+1) \leq f(n+1,m)+f(n,m+1)$. Let S be a set of more than $f(n+1,m)+f(n,m+1)$ points. We will show that it contains an (n+1)-cup or an (m+1)-cap. \\
We partition S into two sets, the "upper" set A and the "lower" set B, according to the following rule. We order the points in S by x coordinate. For any $s=(s_x,s_y)\in S$ and every $s' \in S$ with $s'_x < s_x$, consider the angle defined by the points $s',s,(s_x,s_y+1)$ in order. For $s" \in S$ with $s"_x > s_x$, consider the angle defined by the points $s",s,(s_x,s_y-1)$ in order. If the argument that forms the minimal among the described angles lies to the left of $s$, then $s \in B$, the "lower" set. Otherwise, $s \in A$, the "upper" set. The reader should check that this partition is consistent. \\
Notice that any $n$-cap with nodes exclusively in $A$ can be extended to an (n+1)-cap with rightmost endpoint in $B$. Similarly, any n-cup with nodes exclusively in $B$ can be extended to an (n+1)-cup with leftmost endpoint in $A$. \\
By the pigeonhole principle, either A contains more than $f(n+1,m)$ points, or B contains more than $f(n,m+1)$ points. If $|A|>f(n+1,m)$, then either we have an $(n+1)$-cup, in which case we are done, or we have an $m$-cap with nodes in A, which can be extended to an $(m+1)$-cap with right endpoint in B, in which case we are also done. The other case is treated similarly. So $f(n+1,m+1) \leq f(n+1,m)+f(n,m+1)$. 
We omit the proof that $f(n+1,m+1) \geq f(n+1,m)+f(n,m+1)$, since we will not need it to prove the upper bound. 
\end{proof}

\begin{cor} \textbf{:}
Theorem 1 implies $f(n) \leq {2n-4 \choose n-2}+1$.
\end{cor}

\section{Good Points}

\begin{mydef} \textbf{:}
Given a set of points S and a point $s \in S$, we call the point $(m,l)-$good if, for arbitrary $n \geq 4$, the following holds: By adjoining to S any set B with $|B|\leq n-2$, such that $B\cup S$ contains an $(n-1)$-cup whose left endpoint is s and whose right endpoint is in B, one of the following holds: \\
We either get an l-cap whose 2 rightmost points belong to S,\\
or an m-cup whose 2 leftmost points belong to S,\\
or a convex n-gon . \\
\end{mydef}

\begin{mydef} \textbf{:} 
We call a set of points $(n,m)$-free if it contains no $n$-cup and no $m$-cap.
\end{mydef}

\begin{theorem} \textbf{:}
For $m \geq 4$, any $(m,4)$-free set S of more than ${m \choose 2} - m + 2$ points contains an $(m,4)$-good point s. This point s has the additional property that it is the left endpoint of a $3-$cap.
\end{theorem}

\begin{proof}
Since the set S contains more than ${m \choose 2} -m +2 = {m-1 \choose 2} + 1 $ points, it contains a 4-cap or an (m-1)-cup. In the first case we arrive at a contradiction, since the set is $(m,4)$-free. In the second case, we will start by finding an (m-1)-cup with a specific desirable property. Namely, if p and q are the two rightmost points of the (m-1)-cup (with q being the rightmost one), then we wish that there exists some $r \in S$ that lies to the right of both p and q, such that r lies below the extension of the line pq. Let us denote by R this property of an $(m-1)$-cup. \\ 
We will prove the existence of an $(m-1)$-cup with property R, using induction on m. We start with the base case $m=4$. Given a set $S'$ with 5 points, we divide them into the the "upper" set A and the "lower" set B. If $|A| \geq 3$, then if A contains a 3-cap, we can extend it to a 4-cap using set B, which is a contradiction. Otherwise, it contains a 3-cup. Since the right endpoint of this 3-cup lies in A, this 3-cup must satisfy property R. If $|A| \leq 2$, then $|B| \geq 3$. If B contains a 3-cup, then it can be extended to a 4-cup using A, which is a contradiction. If it contains a 3-cap $v_1,v_2,v_3$, where the nodes are numbered according to the increasing order of their x coordinates, then $\exists a \in A$ such that the 3-cup $a,v_1,v_2$ satisfies property R. So we have proven the base case. \\
Now back to our original set S. Suppose that S does not contain a 4-cap. We will show the existence of an $(m-1)$-cup with the property R. First, we split S into two sets, the "upper" set A and the "lower" set B, as we did before before. If $|A| \geq m - 1$, then it must form an $(m-1)$-cup, otherwise we get a $3-$cap which can be extended to a $4-$cap with right endpoint in B and we are done. But then this $(m-1)$-cup clearly satisfies the desired property R. So in this case we are done. If $|A| \leq m-2 $, then $|B| \geq {m-1 \choose 2} + 2 - (m-2) = {m-2 \choose 2} + 2 $. So by the induction hypothesis, set B contains an $(m-2)$-cup with property R. This can be extended to an $(m-1)$-cup with left endpoint in A, so that property R is preserved. \\
We have proven the existence of an $(m-1)$-cup with property R and are ready to prove the existence of an $(m,4)$-good point s. Let the the (m-1)-cup with property R be denoted by $v_1,v_2,..,v_{m-1}$. The vertices are numbered from right to left, sorted by x coordinate in decreasing order. Set $s=v_2$. Now consider adding an (n-1)-cap $u_1,u_2,..,u_{n-2},u_{n-1}=s$, where again we order the nodes from right to left, such that $u_1$ does not belong to S. In particular, we want $v_1 \neq u_1$. Before we proceed, note again that we picked the (m-1)-cup such that there is a node r that lies to the right of $v_1$ and below the extension of the line $sv_1$. \\
Now we have 4 possible locations for the position of $u_1$ relative to $v_1$: \\
-To the left of $v_1$ and above the line $sv_1$. In this case, one of the following holds:\\
    $u_1,u_2,s=v_2,v_3,..,v_{m-1}$ form an m-cup. \\
    $s,u_{n-2},..,u_2,v_1,u_1$ form a convex n-gon. \\
    $u_3,u_2,v_1,r$ form a 4-cap. \\
-To the left of $v_1$ and below the line $sv_1$. In this case, one of the following holds:\\
    $s,u_{n-2},..,u_2,u_1,v_1$ form an n-cup. \\
    $u_2,u_1,v_1,r$ form a 4-cap. \\
-To the right of $v_1$ and above the extension of the line $sv_1$. In this case: \\
    $v_{m-1},v_{m-2},..,s=v_2,v_1,u_1$ form an m-cup. \\
-To the right of $v_1$ and below the extension of the line $sv_1$. In this case: \\
    $s,v_{n-1},v_{n-2},..,v_2,v_1,u_1$ form a convex n-gon. \\
So s is an $(m,4)$-good point, as it always forces one of the configurations described in Definition 1.
\end{proof}

\begin{mydef} \textbf{:}
Define the family of functions $w_i(m,k)$ for $i\geq 4, m,k \geq 4$ by: \\
$w_4(m,k) = {m+k-6 \choose k-2} + {m+k-7 \choose k-3} + 2 {m+k-8 \choose k-4}$ \\
And for $i \geq 5$: \\
$w_i(m,k) = \begin{cases}
			0 &\text{if } k < i \\
			(i-1) {m+k-i-4 \choose k-i} & \text{if } k \geq i 
			\end{cases} $
\end{mydef}

\begin{theorem} \textbf{:}
For $m \geq 4,l \geq 4$, any $(m,l)$-free set S of more than $g(m,l) = \sum_{i=4}^{\infty} w_i(m,l)$ points contains an $(m,l)$-good point s. This point s also  has the additional property that it is the left endpoint of an $(l-1)-$cap, which we call property $R$.
\end{theorem}

\begin{proof}
Let $h(m,l)$ denote the maximum cardinality of an $(m,l)$-free set that contains no $(m,l)$-good points. \\
Also, $h(m,l)\leq f(m,l)$, since we are concerned with $(m,l)$-free sets. \\
For $l=4$, we showed in Theorem 2 that $h(m,4)\leq {m \choose 2} -m+2$. \\
In light of the identity ${a \choose b} = {a-1 \choose b} + {a-1 \choose b-1}$, we notice that $g(m+1, l+1) = g(m+1,l) + g(m,l+1)$ for $m,l \geq 4$. \\
Moreover, we clearly have $h(m+1,n+1) \leq h(m+1,n) + h(m,n+1)$. \\ 
So if we prove that for $l \geq 4$ we have $h(4,l) \leq g(4,l)$ and that for $m\geq 4$ we have $h(m,4) \leq g(m,4)$, we are done by the above inequalities and strong induction. \\
First, we prove that $h(m,4) \leq g(m,4)$. This follows from \\
$h(m,4) \leq {m \choose 2}-m+2 = \frac{m(m-1)}{2}-m+2 = \frac{(m-2)(m-3)+4m-6}{2}-m+2 = $ \\
${m-2 \choose 2} + (m-3) + 2 = w_4(m,4)$. \\
Next, we prove that $h(4,l) \leq g(4,l)$ by induction on $l$. The base case $l=4$ was just proven. For $l \geq 5$, we have by the inductive hypothesis \\
$h(4,l) \leq h(4,l-1) + h(3,l) \leq g(4,l-1) + f(3,l) = g(4,l-1) + w_l(4,l) = g(4,l)$. The last inequality follows from the fact that $w_i(4,l-1)=w_i(4,l)$ for $l \geq i+1$. \\
$h(m,l) \leq g(m,l)$ follows now by strong induction, so we have proven the theorem.

\end{proof}

\begin{theorem} \textbf{:}
Consider now a set S of $f(n-1,n-1) + g(n,n-2) +1$ points, with $n \geq 6$. Then it contains at least one of the following: \\
1. An n-cup \\
2. An (n-1)-cap \\
3. A convex n-gon 
\end{theorem}

\begin{proof}
We partition S into the "upper" set A and the "lower" set B. If $|B| > f(n-1,n-1)$, then we are clearly done. Otherwise, $|A| = g(n,n-2)+1+(f(n-1,n-1)-|B|)$. Assume that A is $(n,n-2)$-free, otherwise we are done. So now by Theorem 3, A contains at least $f(n-1,n-1)-|B|+1$ points that are (n,n-2)-good, and which are also left endpoints of some $(n-3)$-cap, with vertices in A. Denote this set of points by C. \\
Consider now $B \cup C$. Since $|B \cup C| = f(n-1, n-1) +1$, it contains either an $(n-1)$-cup or an $(n-1)$-cap. In the second case, we are done immediately. In the first we pick an $(n-1)$-cup with vertices in $B\cup C$. If the left endpoint is in B, then we can extend it to an $n$-cup with left endpoint in A. Otherwise let $s \in C$ be the left endpoint of the $(n-1)$-cup. If the right endpoint $t$ is also in C, then we know that there is an $(n-3)$-cap, with vertices in A, whose left endpoint is t. This $(n-3)$-cap can be extended to an $(n-2)$ cap whose right endpoint is in B and its left enpoint remains $t$. So in this case, we have an $(n-1)$-cup whose right endpoint is $t$, and an $(n-2)$-cap whose left endpoint is $t$. So either the first can be extended to an $n$-cup, or the second to an $(n-1)-cap$. So now only one case is remaining. The case that the right endpoint is in B. But in this case, the left endpoint of the $(n-1)$-cup is an $(n,n-2)$-good point in A, while the right endpoint is not in A, so we are instantly done by Definition 2. 

\end{proof}

\section{A Projective Transformation}

Theorem 5 appears in [5],[6]. Below we sketch a proof, which can be found in those papers.

\begin{theorem}[T\'{o}th and Valtr] \textbf{:}
Let $z(n,n-1)$ denote the maximum positive integer, such that there exists a set of cardinality $z(n,n-1)$ that contains no $n$-cup, no $(n-1)$-cap and no convex $n$-gon. Then if the points of a set S form no convex $n$-gon, $|S| \leq z(n,n-1) + 1$.
\end{theorem}

\begin{proof}
Let $L$ denote the let of all straight lines that intersect S at 2 points. Pick some point $x \in S$, that lies on the convex hull of S. Let y be a point outside of the convex hull, such that no line in $L$ intersects the line segment $xy$. Next, pick a line $l$ through y that does not intersect the convex hull (viewed as a convex set) os S. Consider a projective transformation $T$ that maps the line $l$ to the line at infinity and $xy$ to a vertical half-line, emanating from $T(x)$ downwards. The transformation clearly preserves convexity. Also, if $T(S/\{x\})$ contains an $(n-1)$-cap, then this $(n-1)$-cap together with x form an $n$-gon. Since convexity is preserved, the inverse images of $n$ such points form an $n$-gon in the original plane. Similarly, if $T(S/\{x\})$ contains an $n$-cup, then we have a convex $n$-gon in S. So if S contains no convex $n$-gon, then $T(S/\{x\})$ is $(n,n-1)$-free and contains no convex $n$-gon, so $|S|-1 = |S/\{x\}| \leq z(n,n-1) $
\end{proof}

\section{Conclusion}

\begin{theorem} \textbf{:}
$f(n) \leq {2n-6 \choose n-3} + g(n,n-2) + 2$
\end{theorem}

\begin{proof}
In Theorem 4, we proved that $z(n,n-1) \leq f(n-1,n-1) + g(n,n-2)$. So by Theorem 5, the result follows.
\end{proof}

\begin{theorem} \textbf{:} 
$$ \limsup\limits_{n\rightarrow \infty} \frac{f(n)}{{2n-5 \choose n-2}} \leq \frac{29}{32}$$
\end{theorem}

\begin{proof}
First, we notice that $\lim_{n\to\infty} \frac{w_i(n,n-2)}{{2n-5 \choose n-2}} = \frac{i-1}{2^{i+1}}$ for $i \geq 5$ and \\
$\lim_{n\to\infty} \frac{w_4(n,n-2)}{{2n-5 \choose n-2}} \leq \frac{1}{8} + \frac{1}{16} + \frac{2}{32} = \frac{1}{4} $. \\
Now it is simple to argue that $\lim_{n\to\infty} \frac{g(n)}{{2n-5 \choose n-2}} \leq \frac{1}{8} + \frac{1}{16} + \frac{1}{32} + \sum_{i=5}^{\infty} \frac{i-1}{2^{i+1}}=\frac{13}{32}$.
Combining the last inequality with Theorem 6, we get the desired result. 
\end{proof}

\section{Appendix}

\begin{theorem} \textbf{:}
For $m \geq 4,l \geq 5$, any $(m,l)$-free set S of more than $g(m,l) = f(m,l) - {m+l-6 \choose l-3} + {m+l-8 \choose l-5}$ points contains an $(m,l)$-good point s. This point s also  has the property that it is the left endpoint of an $(l-1)-$cap, which we call property $R$.
\end{theorem}

\begin{proof}
Let $h(m,n)$ denote the maximum cardinality of an $(m,l)$-free set that contains no $(m,l)$-good points. \\
For $l=4$, we showed in Theorem 2 that $h(m,4)\leq {m \choose 2} -m+2$. \\
Also we have $h(m,l)\leq f(m,l)$, since any set of more than $f(m,l)$ points is not $(m,l)$-free. \\
Now, by use of the identity ${a \choose b} = {a-1 \choose b} + {a-1 \choose b-1}$, we notice that $g(m+1, l+1) = g(m+1,l) + g(m,l+1)$. \\
Moreover, we clearly have $h(m+1,n+1) \leq h(m+1,n) + h(m,n+1)$. \\ 
So if we prove that for $l \geq 5$ we have $h(4,l) \leq g(4,l)$ and that for $m\geq 4$ we have $h(m,5) \leq g(m,5)$, we are done by the above inequalities and strong induction. \\
We first prove $h(4,l) \leq g(4,l)$ by induction on $l$. The base case $l=5$ holds by $h(4,5)\leq h(4,4)+h(3,5) \leq {4 \choose 2} -4+2 + f(3,5)=6-4+2+4=8$ and $g(4,5) = f(4,5) - {4+5-6 \choose 5-3} + {4+5-8 \choose 5-5} = 10-3+1 = 8$. \\
For $l \geq 6$, we have by the inductive hypothesis $h(4,l) \leq h(4,l-1) + h(3,l) \leq g(4,l-1) + f(3,l) = f(4,l-1) - {4+(l-1)-6 \choose (l-1)-3} + {4+(l-1)-8 \choose (l-1)-5} + f(3,l) = (f(4,l-1) + f(3,l)) -  {l-3 \choose l-4} + {l-5 \choose l-6} = f(4,l) -(l-3)+(l-5) = f(4,l) - (l-2) + (l-4) = f(4,l) - {4+l-2 \choose l-3} + {4+l-4 \choose l-5} = g(4,l)$. \\
Next, we will prove that $h(m,5) \leq g(m,5)$, by showing that $h(m,5) \leq f(m,5)-{m-1 \choose 2} + 1 = g(m,5)$. We will again use induction, this time on $m$. For $m=4$, we have $h(4,5) \leq h(4,4) + h(3,5) \leq {4 \choose 2} -4 +2 + 4 = 8 = 10-3+1 = f(4,5) - {4-1 \choose 2} + 1 =g(4,5) $, which proves the base case. \\
For $m\geq 5$, we have by the inductive hypothesis that: \\
$h(m,5) \leq h(m,4) + h(m-1,5) \leq {m \choose 2} -m+2 + g(m-1,5)=$ \\
${m \choose 2}-m+2 + f(m-1,5) - {m-2 \choose 2} + {m-4 \choose 0} = {m \choose 3} + m$. \\
So we are left to prove that $g(m,5) \geq {m \choose 3}$, which is actually an equality, as it follows from simple algebraic manipulations. \\
$h(m,l) \leq g(m,l)$ follows now by strong induction, so we have proven the theorem.
\end{proof}

We can now use this new function $g$, which can be proven to be equivalent to the function $g$ from Theorem 3, to again obtain the same results mentioned in the main paper. In particular, we get the upper bound:
$$f(n) \leq {2n-5 \choose n-2} - {2n-8 \choose n-5} + {2n-10 \choose n-7}+2$$

\end{document}